\documentclass[11pt,centertags,reqno]{amsart}

\usepackage[foot]{amsaddr}
\usepackage{mathrsfs, mathtools}
\usepackage{latexsym}
\usepackage[english]{babel}
\usepackage[T1]{fontenc}
\usepackage{amssymb}
\usepackage{url}
\usepackage{hyperref}
\usepackage{verbatim}
\usepackage{leftidx}
\mathtoolsset{showonlyrefs}

\numberwithin{equation}{section} \makeatletter
\renewcommand{\subsection}{\@startsection
{subsection}{2}{0mm}{\baselineskip}{-0.25cm}
{\normalfont\normalsize\bf}} \makeatother

\newtheorem{theorem}{Theorem}[section]
\newtheorem{lemma}[theorem]{Lemma}
\newtheorem{corollary}[theorem]{Corollary}
\newtheorem{definition}[theorem]{Definition}
\newtheorem{remark}[theorem]{Remark}
\newtheorem{proposition}[theorem]{Proposition}

\newtheorem{ass}[theorem]{Assumption}

\def \A {\mathcal A}
\def \F {\mathcal F}

\def \H {\mathcal H}

\def \N {\mathcal N}
\def \L {\mathcal L}
\def \P {\mathbf P}
\def \Q {\mathbf Q}
\def \R {\mathbb R}

\def \I {{\mathbf 1}}
\def \bF {\mathbb F}

\def \bH {\mathbb H}
\def \bE {\mathbb E}

\newcommand{\ud}{\mathrm d}
\newcommand{\ds}{\displaystyle}
\newcommand{\esp}[2][\mathbb E] {#1\left[#2\right]}
\newcommand{\espp}[2][\mathbb E^{\mathbf P^*}] {#1\left[#2\right]}

\newcommand{\doleans}[1] {\mathcal E\left(#1\right)}

\begin{document}

\author[C.~Ceci]{Claudia  Ceci}
\address{Claudia  Ceci, Department of Economics,
University ``G. D'Annunzio'' of Chieti-Pescara, Viale Pindaro, 42,
I-65127 Pescara, Italy.  {\it Phone number} {+39 085 453 7579}}\email{c.ceci@unich.it}

\author[K.~Colaneri]{Katia Colaneri}
\address{Katia Colaneri, Department of Economics,
University of Perugia, Via A. Pascoli, 20,
I-06123 Perugia, Italy.{\it Phone number} {+39 075 585 5218}}\email{katia.colaneri@unipg.it}

\author[A.~Cretarola]{Alessandra Cretarola}
\address{Alessandra Cretarola, Department of Mathematics and Computer Science,
 University of Perugia, Via L. Vanvitelli, 1, I-06123 Perugia, Italy. {\it Phone number } {+39 075 585 5021}}\email{alessandra.cretarola@unipg.it}

\title[FS decomposition under incomplete information ]{The F\"ollmer-Schweizer
decomposition under incomplete information}

\maketitle

\date{}

\begin{abstract}
\begin{center}
In this paper we study the F\"ollmer-Schweizer
decomposition of a square integrable random variable $\xi$ with respect to a given semimartingale $S$ under restricted information.
Thanks to the relationship between this decomposition and that of the projection of $\xi$ with respect to the given information flow, we characterize the integrand appearing in the F\"ollmer-Schweizer decomposition under partial information in the general case where  $\xi$ is not necessarily adapted to the available information level.
For partially observable Markovian models where the dynamics of $S$ depends on an unobservable stochastic factor $X$, we show how to compute the decomposition  by means of filtering problems  involving functions defined on an infinite-dimensional space. Moreover, in the case of a partially observed jump-diffusion model where $X$ is described by a pure jump process taking values in a finite dimensional space, we compute explicitly the integrand in the F\"ollmer-Schweizer decomposition by working with finite dimensional filters.
\end{center}
\end{abstract}



\section*{Acknowledgements}
This work was partially supported by the Gruppo Nazionale per l'Analisi Matematica, la Probabilit\`{a} e le loro
Applicazioni (GNAMPA) of the Istituto Nazionale di Alta Matematica (INdAM).

\section{Introduction}

In this paper we study the F\"ollmer-Schweizer decomposition of a square integrable random variable under partial information.
More precisely, we characterize the integrand appearing in this decomposition when the 
random variable is not necessarily adapted to the available information level. Moreover, we discuss an application to a partially observable Markovian model where we compute explicitly this decomposition by means of filtering problems.

It is worth mentioning that the F\"ollmer-Schweizer decomposition has a relevant application in finance. More precisely, under suitable assumptions, the integrand $\beta^\H$ in the decomposition of the square integrable random variable representing the discounted payoff of a given European type contingent claim, provides the locally risk-minimizing hedging strategy in an incomplete financial market driven by semimartingales, see e.g. \cite{fs1991, s01} for more details. There are several papers where the full information case is discussed, see for instance \cite{ms94, ms95, s01, cvv2010}. Results for the partial information setting, can be found in~\cite{ccr2, ccc2}.

In our setting the full information flow is described by a filtration $\bF:=\{\F_t, \ t \in [0,T]\}$, with $T$ denoting a fixed and finite time horizon, while the available information level is given by a smaller filtration $\bH:=\{\H_t, \ t \in [0,T]\}$.

On a given probability space $(\Omega,\F,\P)$, we consider an $(\bF,\P)$-semimartingale $S$ satisfying the structure condition with respect to the full
information flow $\bF$ and a square integrable $\F_T$-measurable random variable $\xi$. The aim is to derive and characterize the F\"ollmer-Schweizer
decomposition of $\xi$ with respect to $S$ under the restricted information given by $\bH$,
that is
\[
\xi = U_0+ \int_0^T \beta^\H_t \ud S_t + A_T,  \quad \P\mbox{-a.s.},
\]
when the following condition on filtrations holds
\[
\F^S_t\subseteq \H_t\subseteq\F_t
\]
where $\F^S_t$ denotes the $\sigma$-algebra generated by $S_t$ up to time $t \in [0,T]$.
Here $U_0$ is a square integrable $\F_0$-measurable random variable, $\beta^\H$ is an $\bH$-predictable $S$-integrable process and $A$ is a square integrable $(\bF, \P)$-martingale orthogonal, in a weak sense, to the martingale part of $S$, see~\cite{ccr2} for more details.
It is worth mentioning that in this decomposition the classical definition of orthogonality between $(\bF,\P)$-martingales is replaced by the concept of $\bH$-weak orthogonality, see~\cite{ccr1, ccr2}.

In~\cite{ccr1} the authors studied the case when $S$ is a
local martingale under the real-word probability measure $\P$
and derived the so-called the Galtchouk-Kunita-Watanabe decomposition under restricted information for a square integrable $\F_T$-measurable random variable. The semimartingale case has been investigated in~\cite{ccc1} under the benchmark approach, which allows for a reduction to the local martingale case, and in~\cite{ccr2}, where the authors provided a version of the F\"ollmer-Schweizer decomposition working under partial information thanks to
existence and uniqueness results for the solution of backward stochastic differential equations driven by a c\`{a}dl\`{a}g  $\bF$-martingale in a partial information framework. However, how to explicitly characterize the integrand appearing in such a decomposition was still an open question.

The case when a random variable $\xi$ is measurable with respect to the smaller $\sigma$-algebra $\H_T$ is studied in~\cite{ccc2}, where the authors characterized the terms in the decomposition under the assumption that $S$ is a quasi-left continuous process. Here,  we extend these results in two directions; precisely, we consider a random variable which is not necessarily $\H_T$-measurable, and we also remove the quasi-left continuity assumption on $S$. Our first achievement is given by Theorem \ref{thm1}, which shows the relationship between the F\"ollmer-Schweizer decomposition of a square integrable $\F_T$-measurable random variable $\xi$ under partial information and that of its projection with respect to $\H_T$, i.e.  $\esp{\xi|\H_T}$. Consequently, we can provide an operative method to compute the integrand $\beta^\H$ in the decomposition above, see  Proposition \ref{casoJ}.

Finally, we discuss an application of these representation results in a Markovian framework where we assume that the underlying semimartingale dynamics is affected by an unobservable stochastic factor $X$. Thanks to the Markov properties of the model and when the available information level coincides with the natural filtration of $S$, we show how to compute the F\"ollmer-Schweizer decomposition of a given square integrable random variable via nonlinear filtering problems, which involve
functions defined on an infinite-dimensional space. Furthermore,
in the case of a partially observable jump-diffusion model where $X$ is described by a pure jump process taking values in a finite dimensional space, we can compute explicitly the integrand in the F\"ollmer-Schweizer decomposition by working with finite-dimensional filters, as shown in Theorem \ref{thm:IntegrandFS}.\\
Filtering results for jump diffusion processes can be found, e.g. in~\cite{FR2010, cco1, GrMi, cco2}. In this paper we use the same technique of~\cite{cco1, cco2} to compute the filter dynamics by means of the innovation approach.

The paper is organized as follows. In Section 2 we describe the incomplete information model. In Section 3 we provide a characterization of the integrand in the F\"ollmer-Schweizer decomposition under partial information of an $\F_T$-measurable square integrable random variable $\xi$. In Section 4 we discuss an application in a Markovian framework and apply filtering arguments to compute explicitly the decomposition.

\section{The incomplete information model}\label{sec:setting}

We fix a probability space $(\Omega,\F,\P)$,  endowed with a filtration $\bF:=\{\F_t, \ t\in [0,T]\}$ that satisfies the usual conditions of
right-continuity and completeness, where $T > 0$ is a fixed and finite time horizon; furthermore, we assume that $\F=\F_T$. On this probability space we consider an $\R$-valued square integrable, c\`{a}dl\`{a}g $(\bF,\P)$-semimartingale $S=\{S_t,\ t \in [0,T]\}$ that satisfies the structure condition (see e.g.~\cite{s01} for further details) given by
\begin{equation} \label{eq:SC}
S_t = S_0 + M_t + \int_0^t\alpha_u^\F  \ud \langle M \rangle_u,\quad t \in [0,T],
\end{equation}
where $S_0 \in L^2(\F_0,\P)$\footnote{Given a $\sigma$-algebra $\mathcal{G}$ and a probability measure $\Q$, the space $L^2(\mathcal G,\Q)$ denotes the set of all $\mathcal G$-measurable random variables $H$ such that $\mathbb E^\Q\left[|H|^2\right] = \int_\Omega |H|^2\ud {\Q} < \infty$.}, $M=\{M_t,\ t \in [0,T]\}$ is an $\R$-valued, square integrable, (c\`{a}dl\`{a}g) $(\bF,\P)$-martingale
starting at null, $\langle M\rangle=\{\langle M,M\rangle_t,\ t \in [0,T]\}$ denotes its $\bF$-predictable quadratic variation process and
$\alpha^\F=\{\alpha_t^\F, \ t \in [0,T]\}$ is an $\R$-valued, $\bF$-predictable process such that $\int_0^T\left(\alpha_s^\F\right)^2 \ud \langle
M\rangle_s < \infty$ $\P$-a.s..

In this setting the restricted information framework is described by an additional smaller filtration $\bH:=\{\H_t, \ t \in [0,T]\}$, i.e.
$$
\H_t \subseteq \F_t, \quad t \in [0,T].
$$
Denote by $\bF^S:=\{\F^S_t, \ t \in [0,T]\}$ the natural filtration of the process $S$, i.e. $\F^S_t=\sigma\{S_u, \ 0 \leq  u \leq t \leq T\}$. We remark that all filtrations are supposed to satisfy the usual hypotheses of completeness and right-continuity.
Now, we make the following assumption
\begin{equation}\label{hp:filtration}
\F^S_t \subseteq \H_t, \quad t \in [0,T].
\end{equation}
This is a sufficient requirement to provide existence of the F\"ollmer-Schweizer decomposition under partial information of a given square integrable $\F_T$-measurable random variable, see \cite[Proposition 3.3]{ccr2}.

In virtue of condition
\eqref{hp:filtration} on filtrations the process $S$ is also an $(\bH,\P)$-semimartingale and therefore it admits a semimartingale decomposition with respect to $\bH$,
i.e.
\begin{equation}\label{eq:Semi}
S_t= S_0+N_t+ R_t, \quad t \in [0,T],
\end{equation}
where $N=\{N_t,\ t \in [0,T]\}$ is  an $\R$-valued, square integrable $(\bH, \P)$-martingale with $N_0=0$ and $R=\{R_t,\ t \in [0,T]\}$ is an
$\R$-valued, $\bH$-predictable process of finite variation with $R_0=0$. In particular, since $R$ is $\bH$-predictable, even this decomposition is unique (see
e.g.~\cite[Chapter III, Theorem 34]{pp}).

We will use the notation 
${}^p X$ 
to indicate the
predictable 
projection with respect
to $\bH$ under $\P$ of a given process $X=\{X_t,\ t \in [0,T]\}$ satisfying $\esp{|X_t|} < \infty$ for every $t \in [0,T]$, defined as the unique
$\bH$-predictable 
process such that  
${}^p X_\tau = \esp{X_{\tau}| \H_{\tau^-}}$ $\P$-a.s. on $\{\tau < \infty\}$ for every
$\bH$-predictable stopping time $\tau$. 

In~\cite{ccc2} the authors proved the structure condition for the semimartingale $S$ with respect to the filtration $\bH$ when $S$ is supposed to be quasi-left-continuous. Here we show that the result holds for any special semimartingale.
To this aim, we introduce the integer-valued random measure associated to the jumps of $S$
\begin{equation*} \label{m}
m(\ud t,\ud z) = \sum_{s: \Delta S_s \neq 0} \delta_{(s, \Delta S_s)}(\ud t,\ud z),
\end{equation*}
where $\delta_a$ denotes the Dirac measure at  point $a$.

Then, we define the characteristics of $S$ (see e.g.~\cite[Chapter II, Section 2]{js}) $(B^\bF, C^\bF, \nu^\bF)$, with respect to $\bF$, and $(B^\bH, C^\bH, \nu^\bH)$, with respect to $\bH$, respectively.  Here $B^\bF = \{\int_0^t \alpha_r^\F \ud \langle M \rangle_r, \ t \in [0,T] \}$ and $B^\bH = R$ denote the $\bF$-predictable and $\bH$-predictable finite variation processes in decompositions \eqref {eq:SC} and \eqref{eq:Semi}, respectively. Moreover, $C^\bF = C^\bH =  \langle M^c\rangle  = \langle N^c\rangle$ (see e.g.~\cite{kxy2006} or~\cite{ms2010} for more details) and  $\nu^{\bF}(\ud t, \ud z)$ and $\nu^{\bH}(\ud t,\ud z)$  are the predictable dual projections of $m(\ud t,\ud z)$ under $\P$ with respect to $\bF$ and
$\bH$ respectively.

According to~\cite[Chapter II, Proposition 2.9]{js}, we can find a version of the characteristics satisfying
$$B^\bF_t = \int_0^t b^\F_s \ud U^\bF_s,  \quad  C^\bF =   \int_0^t c^\F_s \ud U^\bF_s,  \quad \nu^\bF(\ud t, \ud z)= \nu^\bF_t (\ud z) \ud U^\bF_t, \quad t \in [0,T],$$

$$B^\bH_t = \int_0^t b^\H_s \ud U^\bH_s, \quad  C^\bH =   \int_0^t c^\H_s \ud U^\bH_s,  \quad \nu^\bH(\ud t, \ud z)= \nu^\bH_t (\ud z) \ud U^\bH_t\quad t \in [0,T],$$
where $U^\bF$ and $U^\bH$ are  increasing, $\bF$-predictable and $\bH$-predictable processes respectively. Moreover  $U^\bF$ and $U^\bH$ are continuous if and only if $S$ is quasi-left-continuous.

Finally, $b^\F$, $c^\F$ and $b^\H$, $c^\H$ are $\bF$-predictable and $\bH$-predictable processes respectively, and $\nu^\bF_t(\ud z)$ and $\nu^\bH_ t(\ud z)$ are $\bF$-predictable and $\bH$-predictable kernels, respectively.
Note that the characteristics $(B^\bF, C^\bF, \nu^\bF)$ and $(B^\bH, C^\bH, \nu^\bH)$ satisfy
\begin{gather*}
\nu^\bF_t(\{0\})=0,\quad \nu^\bH_t(\{0\})=0, \\
\Delta B^\bF_t=\int z  \nu^\bF(\{t\}, \ud z),\quad \Delta B^\bH_t=\int z  \nu^\bH(\{t\}, \ud z), \\
c^\F=0\ \mbox{ on } \{\Delta U^\bF\neq 0\}, \quad c^\H=0\ \mbox{ on } \{\Delta U^\bH\neq 0\},
\end{gather*}
see e.g.~\cite{cvv2010} for more details.

\begin{lemma}\label{lemma:representationMN}
The martingales $M$ and $N$ have the following representation
\begin{align}
\label{eq:m}M_t &= M^c_t + \int_0^t\int_{\R} z(m(\ud s,\ud z) - \nu^\bF(\ud s,\ud z)), \quad t \in [0,T],\\
\label{eq:n}N_t &= N^c_t + \int_0^t\int_{\R} z(m(\ud s,\ud z) - \nu^\bH(\ud s,\ud z)), \quad t \in [0,T],
\end{align}
where $M^c$ and $N^c$ denote the continuous parts of $M$ and $N$ respectively.
\end{lemma}

\begin{proof}
By~\cite[Chapter II, Corollary 2.38]{js} and the fact that $S$ is a special semimartingale we get that $S$ has the following $(\bF, \P)$-semimartingale representation
\[
S_t=S_0+S^{c,\bF}_t+ \int_0^t\int_{\R} z(m(\ud s,\ud z) - \nu^\bF(\ud s,\ud z))+P_t^\bF, \quad t \in [0,T],
\]
where $P^\bF=\{P_t^\bF,\ t \in [0,T]\}$ is an $\bF$-predictable process of finite variation and $S^{c,\bF}$ is the continuous $\bF$-martingale part of $S$.
By uniqueness of the $(\bF, \P)$-semimartingale decomposition we get \eqref{eq:m}, and in particular $\ds P_t^\bF=\int_0^t\alpha^F_s \ud \langle M\rangle_s$, for each $t \in [0,T]$, in virtue of the structure condition \eqref{eq:SC} for $S$.
Analogously, we get that the $(\bH, \P)$-semimartingale decomposition of $S$ is given by
\[
S_t=S_0+S^{c,\bH}_t+ \int_0^t\int_{\R} z(m(\ud s,\ud z) - \nu^\bH(\ud s,\ud z))+P_t^\bH, \quad t \in [0,T],
\]
for a suitable $\bH$-predictable process $P^\bH=\{P_t^\bH,\ t \in [0,T]\}$ of finite variation, where $S^{c,\bH}$ is the continuous $\bH$-martingale part of $S$. Again, by uniqueness of the $(\bH, \P)$-semimartingale decomposition we get $\eqref{eq:n}$.
\end{proof}

\begin{remark}
It is clear from Lemma \ref{lemma:representationMN} that $S,M$ and $N$ do not have the same jumps. Nevertheless, as the sets of jumps of $M$ and $N$ are subsets of the jumps of $S$, they are $\bF^S$-adapted.
\end{remark}

The predictable quadratic variation processes of $M$ and $N$ are respectively given by
$$
\langle M\rangle_t = \langle M^c\rangle_t +  \int_0^t\int_{\R} z^2\nu^\bF(\ud s,\ud z), \quad t \in [0,T],
$$
$$
\langle N\rangle _t =  \langle M^c\rangle _t + \int_0^t\int_{\R} z^2\nu^\bH(\ud s,\ud z), \quad t \in [0,T].
$$

Note that, the predictable quadratic variations of $M$ and $N$ depend on the filtrations $\bF$ and $\bH$, respectively. However, if
it does not create ambiguity, we will always write $\langle M\rangle = {}^{\bF} \langle M \rangle$ and $\langle N\rangle = {}^{\bH} \langle N \rangle$
to simplify the notation.

In the sequel we denote by $v^{p,\bH}$ the $(\bH,\P)$-predictable dual projection of an $\R$-valued, c\`adl\`ag, $\bF$-adapted process $v=\{v_t,\ t \in [0,T]\}$ of integrable
variation,
defined as
the unique $\R$-valued, $\bH$-predictable process of integrable variation, such that
\begin{equation*}\label{eq:hpredvar}
\esp{\int_0^T\varphi_t\ud v_t^{p,\bH}}=\esp{\int_0^T\varphi_t\ud v_t},
\end{equation*}
for every $\R$-valued, $\bH$-predictable (bounded) process $\varphi=\{\varphi_t, \ t \in [0,T]\}$,
see e.g. Section 4.1 of~\cite{ccr1} for further details.

Then, $S$ also satisfies the structure condition with respect to the filtration $\bH$. More precisely, we have  the following result.
\begin{proposition} \label{prop:N}
Assume that \begin{equation} \label{eq:int_cond}
\esp{\int_0^T\left(\alpha_s^\F\right)^2 \ud \langle M \rangle_s}<\infty.
\end{equation}
Then, under condition \eqref{hp:filtration} the $(\bF,\P)$-semimartingale $S$ satisfies the structure condition with respect to $\bH$, i.e.
\begin{equation*} \label{eq:SC_H}
S_t = S_0 + N_t + \int_0^t \alpha_s^\H \ud  \langle N \rangle_s, \quad t \in [0,T],
\end{equation*}
where $\langle N \rangle$ coincides with the $(\bH,\P)$-predictable dual projection of $ \langle M \rangle$, that is, $\langle N \rangle = \langle
M \rangle^{p,\bH}$ and the $\R$-valued, $\bH$-predictable process $\alpha^\H=\{\alpha_t^\H,\ t \in [0,T]\}$ given by
\begin{equation*} \label{sc1}
\alpha_t^\H:=\frac{\ud \left(\int_0^t \alpha_s^\F \ud \langle M \rangle_s \right)^{p,\bH}}{\ud   \langle M \rangle_s^{p,\bH}}, \quad t \in [0,T],
 \end{equation*}
satisfies an integrability condition analogous to \eqref{eq:int_cond}.
\end{proposition}

\begin{proof}
The proof follows by the same arguments of those of~\cite[Proposition 3.2]{ccc2}. Note that, thanks to Lemma \ref{lemma:representationMN}, here we do not need to require that $S$ has only $(\bF, \P)$-totally inaccessible jump times.
\end{proof}

\begin{remark} \label{rem:pred}
The result of Proposition  \ref{prop:N} implies that for any $\R$-valued, $\bH$-predictable process $\varphi$
satisfying
\begin{equation*}
\esp{\int_0^T\varphi_u^2 \ud \langle M \rangle_u} <\infty,
\end{equation*}
the following equalities hold
$$
\esp{\int_0^t \varphi_s  \alpha_s^\F \ud \langle M\rangle _s} = \esp{\int_0^t \varphi_s \alpha_s^\H \ud \langle N\rangle _s} =\esp{\int_0^t \varphi_s  \alpha_s^\H \ud \langle M\rangle _s}, \quad t \in [0,T].
$$
\end{remark}

\section{The F\"ollmer-Schweizer decomposition under incomplete information}

In this section we study the F\"ollmer-Schweizer decomposition under partial information of an $\F_T$-measurable square integrable random variable in the sense of Definition \ref{def:FS} below. To this aim we introduce the following classes of admissible integrands.

\begin{definition}
 The space $\Theta(\bH)$ (respectively $\Theta(\bF)$)
 consists of all $\R$-valued, $\bH$-predictable (respectively $\bF$-predictable) processes $\theta=\{\theta_t,\ t \in [0,T]\}$ satisfying the following integrability condition
\begin{gather}
\esp{\int_0^T\theta_u^2 \ud \langle N \rangle_u+\left(\int_0^T|\theta_u| |\alpha^\H_u| \ud \langle N \rangle_u\right)^2}<\infty, \\
\left(\mbox{ respectively } \esp{\int_0^T\theta_u^2 \ud \langle M \rangle_u+\left(\int_0^T|\theta_u| |\alpha^\F_u| \ud \langle M \rangle_u\right)^2}<\infty \right).
\end{gather}
\end{definition}

Notice that if  $\theta \in \Theta(\bH)$ (respectively $\theta \in \Theta(\bF)$), the stochastic integral of $\theta$ with respect to $S$, that is $\{\int_0^t \theta_r \ud S_r \ t \in [0,T]\}$, is well defined and turns out to be a square integrable $(\bH,\P)$-semimartingale (respectively $(\bF, \P)$-semimartingale).

For reader's convenience, we recall the notion of $\bH$-{\em weak
orthogonality} between square integrable $(\bF,\P)$-martingales, introduced in~\cite{ccr1}.
\begin{definition}
We say that a square integrable $(\bF,\P)$-martingale $O=\{O_t,\ t \in [0,T]\}$ is $\bH$-{\em weakly
orthogonal} to a square-integrable $(\bF,\P)$-martingale $M=\{M_t,\ t \in [0,T]\}$ if the following condition holds
\begin{equation}\label{eq:orthogcond}
\esp{O_T \int_0^T \varphi_t \ud M_t}=0,
\end{equation}
for all $\R$-valued, $\bH$-predictable processes $\varphi=\{\varphi_t,\ t \in [0,T]\}$ satisfying
\begin{equation*}
\esp{\int_0^T\varphi_u^2 \ud \langle M \rangle_u} <\infty.
\end{equation*}
\end{definition}

 \begin{remark} \label{rem:orth}
Since for any $\bH$-predictable process $\varphi$, the process
$\I_{(0,t]}(s) \varphi_s$, with $t \leq T$,
is $\bH$-predictable,  condition \eqref{eq:orthogcond}
implies that
$$
\esp{O_T \int_0^t \varphi_s \ud M_s}=0,\quad \forall t \in [0,T],
$$
and
by conditioning with respect to $\F_t$, we have
$$
\esp{O_t \int_0^t \varphi_s \ud M_s}= \esp{\int_0^t \varphi_s \ud \langle M, O \rangle _s }= 0, \quad \forall t \in [0,T].
$$
Hence, if $O$ and $M$ are strongly orthogonal (i.e. $\langle M, O \rangle _t=0$ $\P-\mbox{a.s.}$ for every $t \in [0,T]$) then they are also $\bH$-weakly
orthogonal. Moreover, in the case of full information, i.e. $\bH=\bF$,  or when $O$ and $M$ are also $(\bH,\P)$-martingales, the $\bH$-weak orthogonality is equivalent to the strong orthogonality condition (see e.g. Lemma 2 and Theorem 36, Chapter IV, page 180 of~\cite{pp} for a rigorous proof).
\end{remark}

The definition of the F\"{o}llmer-Schweizer decomposition under partial information of a given square integrable random variable provided in~\cite{ccr2}, is given below.

\begin{definition}\label{def:FS}
Given $\xi \in L^2(\F_T, \P)$, we say that $\xi$ admits a {\em F\"{o}llmer-Schweizer decomposition under partial information} with respect to $\bH$ and $S$,  if there exist a random variable $U_0 \in L^2(\F_0, \P)$, an $\bH$-predictable process $\beta^\H\in \Theta(\bF)$ and a square integrable $(\bF,\P)$-martingale $A=\{A_t,\ t \in [0,T]\}$  with $A_0=0$,
 $\bH$-weakly
orthogonal to the $\bF$-martingale part $M$ of $S$,  such that
\begin{equation}\label{eq:fsdecomposition}
\xi = U_0 + \int_0^T \beta^\H_t \ud S_t + A_T \quad \P-\mbox{a.s.}.
\end{equation}
\end{definition}

In the sequel, we make the following assumptions.
 \begin{ass} \label{ass:existence}
\begin{itemize}
\item[]
\item[(i)]
There exists a deterministic function $\rho:\R_+ \rightarrow \R_+$ with $\rho(0^+)=0$ such that, $\P$-a.s.,
\begin{equation} \label{ass:bracket}
\langle M\rangle_t - \langle M\rangle_s \leq \rho(t-s), \quad \forall\ 0 \leq s \leq t \leq T.
\end{equation}
\item[(ii)] There exists a constant $k\geq 0$ such that
\begin{equation} \label{bound2}
|\alpha_t^\F| \leq k S_t\quad (\P \otimes \langle M\rangle)-\mbox{a.e.\ on}\ \Omega \times [0,T].
\end{equation}
\end{itemize}
\end{ass}
In~\cite{ccr2}, it is proved that decomposition \eqref{eq:fsdecomposition} exists and it is unique under either condition \eqref{hp:filtration} and Assumption \ref{ass:existence} or  condition \eqref{ass:bracket} and the existence of a constant $c\geq 0$ such that
\begin{equation*} \label{bound1}
|\alpha_t^\F| \leq c, \quad (\P \otimes \langle M\rangle)-\mbox{a.e.\ on}\ \Omega \times [0,T].
\end{equation*}

Let $\xi \in L^2(\F_T, \P)$. Now, the problem is how to compute the integrand in decomposition \eqref{eq:fsdecomposition}. If the process $S$ has continuous trajectories, the integrand can be calculated
by switching to a particular martingale measure $\P^*$, the so-called {\em minimal martingale measure}, and computing the
Galtchouk-Kunita-Watanabe decomposition of $\xi$ with respect to $S$ under $\P^*$. Concerning the more general case, that is, when $S$ is only c\`{a}dl\`{a}g, there are few results in literature, see e.g.~\cite{cvv2010} for the complete information case and~\cite{ccc2} for the incomplete information setting when $\xi$ is $\H_T$-measurable.

Here, the idea is to work with the projection of $\xi$ with respect
to $\H_T$, i.e.
\begin{equation} \label{claimproie}
{}^o\xi := \esp{ \xi | \H_T}.
\end{equation}
By Jensen's inequality we get that ${}^o\xi \in L^2(\H_T,\P)$ and since $S$ is $\bH$-adapted in virtue of condition \eqref{hp:filtration}, under suitable assumptions (see Remark   \ref{tradeoff}  below) ${}^o\xi$
admits a (classical) F\"ollmer-Schweizer
decomposition with respect to $S$ and $\bH$, that is, there exist a random variable $\tilde U_0 \in L^2(\H_0, \P)$, a process $\tilde \beta^\H\in \Theta(\bH)$ and a square integrable $(\bH,\P)$-martingale $\tilde A=\{\tilde A_t,\ t \in [0,T]\}$  with $\tilde A_0=0$, which is
 $\bH$-strongly orthogonal to the $\bH$-martingale part $N$ of $S$, such that
 \begin{equation}\label{eq:fsdecomposition2}
{}^o\xi = \tilde U_0 + \int_0^T \tilde \beta^\H_t \ud S_t + \tilde A_T \quad \P-\mbox{a.s.}.
\end{equation}

\begin{remark} \label{tradeoff}
A sufficient condition for existence and uniqueness of
decomposition \eqref{eq:fsdecomposition2} is the uniform boundedness of the mean-variance tradeoff process $K := \{\int_0^t (\alpha^\H_u) ^2 \ud \langle N \rangle _u,\ t \in [0,T]\}$ in $t$ and $\omega$ (see~\cite[Theorem 3.4]{ms95}) or the fulfillment by $\langle N\rangle $ and $\alpha^\H$
of Assumption \ref{ass:existence}.
\end{remark}

In the sequel we prove that the integrands in decompositions  \eqref{eq:fsdecomposition}  and \eqref{eq:fsdecomposition2} coincide, that is, $ \beta^\H_t  = \tilde \beta^\H_t $, $\P$-a.s., for each $t\in [0,T]$. As a consequence, we can compute $\beta^\H$ in terms of the Galtchouk-Kunita-Watanabe decomposition of ${}^o\xi$  with respect to $S$ and $\bH$ under the minimal martingale measure $\P^*$ (if it exists), see Definition \ref{def:MMM} and Proposition \ref{casoJ} below.

\begin{theorem}\label{thm1}
Let $\xi \in  L^2(\F_T,\P)$. Under condition \eqref{hp:filtration} and Assumption \ref{ass:existence}, the random variable $\xi$ admits the F\"ollmer-Schweizer decomposition under partial information given in \eqref{eq:fsdecomposition}, and ${}^o\xi$ admits  a F\"ollmer-Schweizer decomposition \eqref{eq:fsdecomposition2} with
$$
\tilde U_0=  \esp{U_0 |\H_0} , \quad \tilde \beta^\H_t =  \beta^\H_t \quad \tilde A_t= \esp{A_t|\H_t} +  \esp{U_0 |\H_t}  - \tilde U_0, \quad t \in [0,T].
$$
\end{theorem}

\begin{proof}
Recall that under condition \eqref{hp:filtration} and Assumption \ref{ass:existence}, the random variable $\xi$  admits the F\"ollmer-Schweizer decomposition under partial information given in \eqref{eq:fsdecomposition} thanks to the results of~\cite{ccr2}.
By \eqref{claimproie} and conditioning equation \eqref{eq:fsdecomposition} with respect to $\H_T$,  we obtain that
$$
{}^o\xi = \esp{U_0|\H_T} + \int_0^T \beta^\H_t \ud S_t + \esp{A_T|\H_T} \quad \P-\mbox{a.s.}.
$$
Set
 $$
 \hat A_t := \esp{A_t|\H_t} +  \esp{U_0 |\H_t}  -  \esp{U_0 |\H_0},
 $$
 for every $t \in [0,T]$.
Then
 $$
 {}^o\xi =  \esp{U_0 |\H_0} + \int_0^T \beta^\H_t \ud S_t + \hat A_T \quad \P-\mbox{a.s.}.
 $$
It is easy to verify that
$\hat A=\{\hat A_t, \ t \in [0,T]\}$ is an $(\bH,\P)$-martingale such that $\hat A_0=0$.  If we prove that that $\beta^\H\in \Theta(\bH)$ and $\hat A$ is strong orthogonal to $N$, we obtain the thesis with the choice $\tilde U_0=  \esp{U_0 |\H_0}$,  $\tilde \beta^\H=  \beta^\H$ and $\tilde A= \hat A$.

For the first part, we prove that every $\bH$-predictable process in $\Theta(\bF)$ also belongs $\Theta(\bH)$. Let $\theta$ be an $\bH$-predictable process in $\Theta(\bF)$. Then, relationship $\langle N\rangle = \langle M\rangle^{p,\bH}$ implies
\begin{equation*} \label{f1}
\esp{\int_0^T \theta_s^2 \ud \langle N\rangle_s }=
\esp{\int_0^T \theta_s^2 \ud \langle M\rangle_s}.
\end{equation*}
On the other hand, we get
\begin{equation*} \label{f2}
\esp{ \int_0^T (\theta_s\alpha^\H_s)^2 \ud  \langle N\rangle_s } \leq \esp{ \int_0^T (\theta_s\alpha^\F_s)^2 \ud  \langle M\rangle_s }.
\end{equation*}
Then, thanks to Remark \ref{rem:pred} with the choice $\varphi_s = \theta^2_s\alpha^\H_s$ and the Cauchy-Schwarz inequality, we have
$$
\esp{ \int_0^T (\theta_s\alpha^\H_s)^2 \ud  \langle N\rangle_s } = \esp{ \int_0^T \theta^2_s\alpha^\H_s \alpha^\F_s \ud  \langle M\rangle_s } \leq \left( \esp{ \int_0^T (\theta_s\alpha^\H_s)^2 \ud  \langle N\rangle_s } \right) ^{\frac{1}{2}} \left( \esp{ \int_0^T (\theta_s\alpha^\F_s)^2 \ud  \langle M\rangle_s } \right) ^{\frac{1}{2}}.
$$
Hence, $\theta \in \Theta(\bH)$.
As concerns strong orthogonality,  taking  Remark \ref{rem:orth} into account we show that condition
\begin{equation} \label{c1}
\esp{ \hat A_T \int_0^T \varphi_s \ud N_s}=0
\end{equation}
holds for all $\R$-valued, $\bH$-predictable processes $\varphi$ satisfying
\begin{equation*}
\esp{\int_0^T\varphi_u^2 \ud \langle M \rangle_u} <\infty.
\end{equation*}
First, we get that
$$
\esp{ \hat A_T \int_0^T\varphi_s \ud N_s} = \esp{  A_T \int_0^T \varphi_s \ud N_s} + \esp{ \eta \int_0^T \varphi_s \ud N_s} =
\esp{  A_T \int_0^T \varphi_s \ud N_s},
$$
where $\eta := \esp{U_0 |\H_T}  - \esp{U_0|\H_0}$ and  $\{\int_0^t \varphi_s \ud N_s,\ t \in [0,T]\}$  is an $(\bH,\P)$-martingale null at $t=0$.

Finally, since $A$ is $\bH$-weakly orthogonal to $M$, by the properties of the predictable projection we have
\begin{align*}
\esp{  A_T \int_0^T \varphi_s \ud N_s} & = \esp{  A_T \int_0^T \varphi_s \ud M_s} + \esp{  A_T \int_0^T \varphi_s ( \alpha^\F_s \ud \langle M \rangle _s  -  \alpha^\H_s \ud \langle N \rangle _s)}\\
& = \esp{\int_0^T  A_{s^-} \varphi_s ( \alpha^\F_s \ud \langle M \rangle _s  -  \alpha^\H_s\ud \langle N \rangle _s)} = 0,
\end{align*}
where the last equality follows by Remark \ref{rem:pred}, and this yields \eqref{c1}.
\end{proof}

It is helpful to recall the definition of the minimal martingale measure with respect to the filtration $\bF$.
\begin{definition} \label{def:MMM}
An equivalent martingale measure $\P^*$ for $S$ with square integrable density $\ds \frac{\ud \P^*}{\ud \P}$ is called
{\em minimal martingale measure} (for $S$) if $\P^*=\P$ on $\F_0$ and if every square integrable $(\bF, \P)$-martingale, strongly orthogonal to the
 $\bF$-martingale part $M$ of $S$, is also an $(\bF, \P^*)$-martingale.
\end{definition}
We assume that
  \begin{equation}\label{eq:hpAS2}
  1-\alpha_t^\F \Delta M_t>0 \quad \P-\mbox{a.s.} \quad \forall \ t \in [0,T],
  \end{equation}
and
 \begin{equation}\label{eq:square_integrability_L}
\esp{\exp\left\{\frac{1}{2}\int_0^T \left(\alpha_t^\F\right)^2 \ud \langle M^c\rangle_t+ \int_0^T \left(\alpha_t^\F\right)^2 \ud \langle M^d \rangle_t
\right\}} < \infty,
  \end{equation}
where $M^d$ denotes the discontinuous part of the $(\bF,\P)$-martingale $M$, and define the process $L=\{L_t,\ t
  \in [0,T]\}$ by setting
  \begin{equation*}\label{eq:MMMpartial}
L_t:=\doleans{-\int \alpha_u^\F \ud M_u}_t,\quad t \in [0,T],
\end{equation*}
where the notation $\mathcal{E}(Y)$ refers to the Dol\'{e}ans-Dade exponential of an $(\bF,\P)$-semimartingale $Y$.
 Assuming $L$ to be square integrable, under conditions \eqref {eq:hpAS2} and \eqref {eq:square_integrability_L}  by the Ansel-Stricker Theorem (see~\cite{AS}) there exists the minimal martingale measure $\P^*$ for $S$, which is defined by
\begin{equation}\label{eq:MMMpartial_def}
L_t=\left.\frac{\ud \P^*}{ \ud \P}\right|_{\F_t}\quad t \in [0,T].
\end{equation}


In the rest of the paper we also make the following assumption.
\begin{ass}\label{ass:pstar}
Assume that ${}^o\xi \in L^2(\H_T,\P^*)$ and $S_t\in L^2(\H_t,\P^*)$, for each $t \in [0,T]$.
\end{ass}
Since $S$ is an $(\bH, \P^*)$-martingale, the random variable ${}^o\xi$ admits the
Galtchouk-Kunita-Watanabe decomposition with respect to $S$ and $\bH$ under $\P^*$, i.e.
\begin{equation}\label{GKW2}
{}^o\xi =  \widehat U_0 + \int_0^T  H^\H_u \ud S_u +G_T \quad \P^*-\mbox{a.s.},
\end{equation}
where $\widehat U_0 \in L^2(\H_0,\P^*)$, $H^\H=\{H_t^\H, \ t \in [0,T]\}$ is an $\R$-valued  $\bH$-predictable
process satisfying $\espp{\int_0^T(H_u^{\H})^2 \ud \langle S\rangle_u}<\infty$, and $G=\{G_t, \ t \in [0,T]\}$  is  a square-integrable $(\bH,\P^*)$-martingale with $G_0=0$, strongly orthogonal to $S$ under $\P^*$.

Define the process $V^\H=\{V_t^\H, \ t \in [0,T]\}$ by setting
\begin{equation}\label{def:vh}
V_t^\H:= \bE^{\P^*}[ {}^o\xi  | \H_t],\quad t \in [0,T].
\end{equation}
By decomposition \eqref{GKW2}, we get that
\begin{equation}\label{GKWV}V_t^\H = \widehat U_0 + \int_0^t  H^\H_u \ud S_u +G_t \quad \P^*-\mbox{a.s.};
\end{equation}
hence, we can compute $H^\H$ as
\begin{equation}\label{eq:strategia_filtro1}
H^\H_t=\frac{\ud {}^{*,\bH}\langle V^\H, S\rangle_t}{\ud  {}^{*,\bH} \langle S\rangle_t}, \quad t \in [0,T],
\end{equation}
where ${}^{*,\bH}\langle \cdot \rangle$ denotes the sharp bracket computed with respect to $\bH$ and $\P^*$.

The following result characterizes the process $\beta^\H$.

\begin{proposition}\label{casoJ}
Let $\xi \in  L^2(\F_T,\P)$. Under Assumptions \ref{ass:existence} and \ref{ass:pstar},
the integrand $\beta^\H$ in the F\"ollmer-Schweizer decomposition under partial information \eqref{eq:fsdecomposition} is given by
\begin{equation}\label{eq:betaH}
\beta^\H_t=H^\H_t+\phi^\H_t, \quad  \P-\rm{a.s.} \quad t \in [0,T],
\end{equation}
where $H^\H$ is the integrand in the Galtchouk-Kunita-Watanabe decomposition of ${}^o\xi$  with respect to $S$ and $\bH$ under $\P^*$,  see \eqref{GKW2}, given in \eqref{eq:strategia_filtro1} and
\begin{equation}\label{phiH}
\quad \phi^\H_t =\frac{ \ud {}^\bH\langle [G, S],  \int_0^\cdot \alpha^\H_r \ud N_r \rangle_t }
{ \ud {}^\bH\langle S\rangle_t} , \quad \P-\mbox{a.s.} \quad t \in [0,T],
\end{equation}
where the sharp brackets ${}^{\bH}\langle \cdot \rangle$ are computed with respect to $\bH$ and $\P$\footnote{The $\bF$-predictable (respectively $\bH$-predictable) quadratic variation of the semimartingale $S$, denoted by $\langle S \rangle$ (respectively ${}^{\bH}\langle S \rangle$), is the $\bF$-predictable (respectively $\bH$-predictable) compensator of the quadratic variation process $[S]$.}.
\end{proposition}

\begin{proof}
By Theorem \ref{thm1} we get that $\beta^\H_t=\widetilde \beta^\H_t$ for every $t \in [0,T]$. Then, by formula (4.6) in~\cite[Proposition 4.8]{ccc2} we get \eqref{eq:betaH}, where $H^\H$ is given by \eqref{eq:strategia_filtro1}. Finally, the characterization of $\phi^\H$ follows by Theorem 3.2 and Remarks on page 860 in~\cite{cvv2010}.
\end{proof}

\begin{corollary}\label{totally inaccessible jumps}
Under the hypotheses of Proposition \ref{casoJ}, assume that $S$ has only $(\bF, \P)$-totally inaccessible jump times, then
\begin{equation}\label{eq:strategia_filtro}
\beta^\H_t= H^\H_t+\phi^\H_t,\quad \P-\mbox{a.s.} \quad t \in [0,T].
\end{equation}
with $H^\H$ given by \eqref{eq:strategia_filtro1} and
\begin{equation}\label{phiH2}
\quad \phi^\H_t = \frac{ \ud {}^\bH\langle [G, S],  \int_0^\cdot \alpha^\H_r \ud N_r \rangle_t }
{ \ud {}^\bH\langle N \rangle_t} , \quad \P-\mbox{a.s.} \quad t \in [0,T],
\end{equation}
where the sharp brackets are computed under $\P$.
\end{corollary}

\begin{proof}
Note that the $\bH$-predictable quadratic variation of $S$ is given by ${}^\bH\langle S \rangle={}^\bH\langle N\rangle+ \sum_{s\leq t}(\Delta R_s)^2$. Then, the result follows by Proposition \ref{casoJ} and the observation that, if S has $(\bF, \P)$-totally inaccessible jump times, then the finite variation parts in the semimartingale decompositions of $S$ with respect to both filtrations $\bF$ and $\bH$ are continuous. This implies that ${}^\bH\langle N\rangle={}^\bH\langle S \rangle$ and the expression of \eqref{phiH} reduces to $\eqref{phiH2}$.
\end{proof}

Note that representation \eqref{eq:strategia_filtro} shows how the knowledge about decomposition \eqref{GKWV} of $V^\H$ is an essential tool to compute $\beta^\H$. In the next section we discuss an application in a Markovian framework.

\section{Application to partially observable Markov models}

We consider a partially observable Markovian model, where the dynamics of the semimartingale $S$ is affected by  an unobservable external factor, modeled by  a Markov process denoted by $X$ and such that  the pair $(X, S)$ turns out to be an $(\bF,\P)$-Markov process.

Here we assume that the available information coincides with the natural filtration of the semimartingale $S$, precisely,  $\bH=\bF^S$. Then, we consider a random variable $\xi \in L^2(\F_T,\P)$  of the form
\begin{equation}\label{eq:claim}
\xi=  H(T, X_T, S_T),
\end{equation}
where $H(t,s,x)$ is a given deterministic function.

The goal of this section is to characterize the integrand in the F\"ollmer-Schweizer decomposition under partial information of the random variable $\xi$ given in \eqref{eq:claim}. To this aim, all the hypotheses made in the previous sections are assumed to be fulfilled.

In order to apply Proposition \ref{casoJ} to compute $\beta^\H$, it is essential to provide a representation for the projection of $\xi$ with respect to $\H_T$, i.e.
\begin{equation}\label{eq:oxi}
{}^o\xi = \bE[ H(T,S_T,X_T)  | \H_T].
\end{equation}
Under the hypothesis that $\H_t=\F^S_t$, for every $t \in [0,T]$, the projection ${}^o\xi$ can be written in terms of the filter $\pi$ with respect to $\P$, $\pi(f)=\left\{\pi_t(f), \ t \in [0,T]\right\}$, defined by setting
\begin{equation} \label{def:filtro}\pi_t(f) : =  \mathbb{E}[ f(t,X_t,S_t)  | \F^S_t ] = \int_\R f(t,x,S_t) \pi_t(\ud x), \ \forall t \in [0,T],\end{equation}
for any measurable function $f(t,x,s)$ such that $\mathbb{E} |f(t,X_t,S_t)|< \infty$, for every $t\in [0,T]$. The filter is a probability measure valued  process, which provides the conditional law of the stochastic factor $X$ given the information flow.
In particular we get that
\begin{equation}\label{eq:claimp}
{}^o\xi = \bE[ H(T,S_T,X_T)  | \H_T] = \pi_T(H).
\end{equation}
We denote by ${\mathcal P}(\R)$ the space of probability measures on $\R$ with the weak topology and denote by $p$ its elements.

In the sequel we assume the vector $(X,S, \pi)$ to be an $(\bF, \P^*)$-Markov process and we denote by $\L^*$ its generator.

\begin{remark}
This is a natural assumption satisfied by a large class of models where the change of probability measure, defining the minimal martingale measure $\P^*$ is Markovian. We  provide an example in Section \ref{sec:jumpdiff_model}.
\end{remark}

By~\cite[Chapter 4, Proposition 1.7]{ek86} we get that for every function $f(t,x,s,p)$, with $(t,x,s,p) \in [0,T] \times \R^2 \times {\mathcal P}(\R)$, in the domain of the operator $\L^*$, $D(\L^*)$,  the process $m^f=\{m^f_t, \ t \in [0,T]\}$ given by
\begin{equation} \label{GEN}
m^f_t=f(t,X_t, S_t, \pi_t) -\int_0^t\L^{*} f(u,X_u, S_u, \pi_u) \ud u,\quad t \in [0,T]
\end{equation}
is an $(\bF, \P^*)$-martingale. This leads to the following result.
\begin{proposition}\label{VH}
Let the vector $(X,S, \pi)$ be  an $(\bF, \P^*)$-Markov process.
Then, the   process $V^\H$ given in \eqref{def:vh} admits the following representation
\begin{equation}\label{eq:vh}
V_t^\H = \bE^{\P^*}\left[ g(t, X_t, S_t, \pi_t)   \Big| \H_t\right] \quad t \in [0,T],
\end{equation}
where $g(t, x, s, p)$ is a measurable function on  $[0,T] \times \R^2  \times {\mathcal P}(\R)$ such that
 \begin{align} \label{g}
 g(t, X_t, S_t, \pi_t) = \bE^{\P^*}\left[ \pi_T(H)  \Big| \F_t\right].
 \end{align}
\end{proposition}

\begin{proof}
Since $ \pi_T(H)= \int_\R H(T,x,S_T) \pi_T(\ud x) $ depends on $\omega$ through $(S_T(\omega), \pi_T(\omega))$, then, by the $(\bF, \P^*)$-Markov property  of the triplet $(X,S, \pi)$, we get
\begin{align*}\label{V1}
V_t^\H & = \bE^{\P^*}[ {}^o\xi  | \H_t]  =   \bE^{\P^*}[  \pi_T(H)  | \H_t]  =   \bE^{\P^*}[   \bE^{\P^*}[  \pi_T(H)  | \F_t]   | \H_t] =  \bE^{\P^*}\left[ g(t, X_t, S_t, \pi_t)   \Big| \H_t\right]
\end{align*}
where $g(t, x, s, p)$, is a measurable function of its arguments $(t, x, s, p) \in [0,T] \times \R^2  \times {\mathcal P}(\R)$, such that \eqref{g} is fulfilled.
\end{proof}
The next lemma gives a representation of the function $g$ as a solution to a problem with final condition.
\begin{lemma}\label{lemma:caratterizzazioneG}
Let $\widetilde{g}(t, x, s, p)\in D(\L^*)$  such that
\begin{equation}\label{eq:problema}
\left\{
\begin{aligned}
\L^* \widetilde{g}(t, x,s, p) &=0, \quad (t,x,s,p) \in [0,T)\times \R^2\times \mathcal{P}(\R)\\
\widetilde{g}(T, x, s, p)&= p(H) = \int_{\R} H(T,y,s) p(\ud y), \quad (x,s,p)\in \R^2\times \mathcal{P}(\R).
\end{aligned}
\right.
\end{equation}
Then, $\widetilde{g}(t, X_t,S_t, \pi_t)= g(t,X_t,S_t, \pi_t)$ $\P$-a.s., for every $t \in [0,T]$.
\end{lemma}

\begin{proof}
Let $\widetilde{g}(t, x, s, p) \in D(\L^*)$ be the solution of \eqref{eq:problema}. Then, by \eqref{GEN},  we get that the process $\ds \left\{\widetilde{g}(t, X_t, S_t, \pi_t), \ t \in [0,T]\right\}$
is an $(\bF, \P^*)$-martingale with final value $\widetilde{g}(T, X_T, S_T, \pi_T)=\pi_T(H)$. As a consequence
$\widetilde{g}(t, X_t, S_t, \pi_t)=\bE^{\P^*}[ \pi_T(H) | \F_t]$, for every $t \in [0,T]$, which implies the thesis.
\end{proof}

In general, the vector process $(X,S, \pi)$ takes values in the infinite dimensional space $\R^2 \times {\mathcal P}(\R)$. However, if $X$  assumes finitely many values, then the filter $\pi$ is also finite, and this reduces an infinite dimensional state problem to a finite dimensional one.

Precisely, we assume that $X$ takes values in a set $\mathcal D=\{x_1,\dots, x_d\}$, with $x_i\in \R$ for every $i=1,\dots,d$. Then, for any function $f(t,x,s)$ we can write
\begin{equation} \label{pi_i}
\pi_t(f) = \sum_{i=1}^{d} \bE[ f(t, X_t, S_t) \mathbf{1}_{\{X_t=i\}} |  \H_t] = \sum_{i=1}^{d} f(t, x_i, S_t) \pi_t(f_i),
\end{equation}
where $f_i(x)=\I_{\{x=x_i\}}$  and $ \pi_t(f_i)= \P( X_t=i | \H_t )$. Then, the filter is completely characterized by the conditional probabilities $\pi_t(f_i)$, $i=1,\dots,d$ for every $t\in[0,T]$. Denoting the conditional probabilities vector by $\underline{\pi}:=\{\underline{\pi_t}, \ t \in [0,T]\}$ with $\underline{\pi_t}:= (\pi_t(f_1), \dots, \pi_t(f_d))$, we get that the process $(X,S, \underline{\pi})$ takes values in the finite dimensional space $\R^2 \times [0,1]^d$.

In this framework, since $(X,S, \underline{\pi})$ is an $(\bF, \P^*)$-Markov process with generator $\L^*$, relationship \eqref{eq:vh} in Proposition \ref{VH} can be written as
$$
V_t^\H = \bE^{\P^*}\left[ g(t, X_t, S_t, \underline{\pi_t})   \Big| \H_t\right] \quad t \in [0,T],
$$
where, $g(t, x, s, \underline{p})$ now denotes a measurable function on  $[0,T] \times \R^2 \times [0,1]^{d}$ such that
\begin{equation}
 g(t, X_t, S_t,  \underline{\pi_t}) = \bE^{\P^*}\left[ \pi_T(H)  \Big| \F_t\right] =  \bE^{\P^*}\left[ \sum_{i=1}^{d} H(T,  x_i, S_T) \pi_T(f_i) \Big| \F_t\right],
 \end{equation}
for every $t \in [0,T]$. Similarly to Lemma \ref{lemma:caratterizzazioneG}, we can characterize the function $g(t,x,s, \underline{p})$ as follows.

\begin{lemma}\label{lemma:caratterizzazioneG1}
Let $\widetilde{g}(t, x, s, \underline{p})\in D(\L^*)$  such that
\begin{equation}\label{eq:problema1}
\left\{
\begin{aligned}
\L^* \widetilde{g}(t, x,s, \underline{p}) &=0, \quad (t,x,s,\underline{p}) \in [0,T)\times \R^2\times [0,1]^d \\
\widetilde{g}(T, x, s, \underline{p})&= \sum_{i=1}^{d} p_i H(T, x_i, s), \quad (x,s,\underline{p}) \in  \R^2\times [0,1]^d.
\end{aligned}
\right.
\end{equation}
Then, $\widetilde{g}(t, X_t,S_t,  \underline{\pi_t})= g(t,X_t,S_t, \underline{\pi_t})$ $\P$-a.s., for every $t \in [0,T]$.
\end{lemma}

\begin{proof}
The proof follows from the same lines of those of Lemma \ref{lemma:caratterizzazioneG} and equation \eqref{pi_i}, which characterizes the filter for a state process assuming finitely many values.
\end{proof}

\subsection{A partially observable  jump-diffusion model}\label{sec:jumpdiff_model}
In this section we discuss a partially observable model where
  $S$ is described by a geometric jump diffusion process, with drift and jump characteristics depending on an unobservable stochastic factor $X$,  modeled by a pure jump Markov process having common jump times with $S$ and taking values in the finite space ${\mathcal D} = \{x_1,\dots, x_d\}$, with $x_i \in \R$ for every $i=1,\dots,d$. Precisely, we consider the following system of SDEs
\begin{equation}\label{eq:sistema}
\left\{
\begin{aligned}
\ud X_t&=  \int_Z K_0(\zeta; t, X_{t^-}) \N(\ud t, \ud \zeta), \quad X_0= x_0 \in \mathcal{D}, \\
\ud S_t&=S_{t^-}\left(\mu_1(t, X_t, S_t)\ud t+\sigma_1(t,S_t)\ud W^1_t+\int_Z K_1(\zeta;t, X_{t^-}, S_{t^-}) \N(\ud t, \ud \zeta)\right),\quad
S_0=s_0>0,
\end{aligned}
\right.
\end{equation}
where $\N(\ud t, \ud \zeta)$ is an  $(\bF, \P)$-Poisson random measure with finite intensity  $\eta(\ud \zeta)\ud t$ on a measurable and separable space $(Z, \mathcal Z)$,  $W^1:=\{W^1_t, t \in [0,T]\}$ is an $(\bF,\P)$-Brownian motion independent of $\N(\ud t, \ud \zeta)$, the coefficients $\mu_1(t,x,s)$, $ \sigma_1(t,s)>0$, $K_0(\zeta;t,x)$ and $K_1(\zeta; t,x,s)$ are
$\R$-valued measurable functions of their arguments such that a unique strong solution for the system \eqref{eq:sistema} exists, see for
instance~\cite{OS}.

Here, $K_0(\zeta;t,x)$ takes values in the set $\mathcal K_0 :=\{ k_{ij} = x_i - x_j:  i \neq j, \ i,j =1,...d \}$.

Denote by $\widetilde \N(\ud t, \ud \zeta)$ the $(\bF, \P)$-compensated random measure given by
\begin{equation*} \label{def:cm}
\widetilde \N(\ud t, \ud \zeta)=\N(\ud t, \ud\zeta)-\eta(\ud \zeta)\ud t.
\end{equation*}
We recall the definition of the integer-valued random measure associated to the jumps of $S$ given by
$$
m(\ud t,\ud z) = \sum_{s: \Delta S_s \neq 0} \delta_{(s, \Delta S_s)}(\ud t,\ud z).
$$
For every $t \in [0,T]$, set $D_t:=\{\zeta \in Z: K_1(\zeta;t, X_{t^-}, S_{t^-})\neq 0\}$ and assume $\esp{\int_0^T \eta(D_t) \ud t} < \infty$. Then the  $(\bF, \P)$-predictable dual projection of  $m(\ud t,\ud z)$ is  given by \[\nu^\bF(\ud t,\ud z) = \nu^\bF_t(\ud z) \ud t, \]
where $\nu^\bF_t(\mathcal A) = \eta(D_t^{{\mathcal A}})$, $D^{\mathcal A}_t:=\{\zeta \in Z: K_1(\zeta;t, X_{t^-}, S_{t^-})\in  {\mathcal A} \setminus \{0\} \}$, for any ${\mathcal A} \in {\mathcal B}(\R)$ (see e.g. \cite{cg06,cco1} for more details).

Notice that $\nu^\bF_t(\ud z)$ depends on $\omega$ through $(X_{t^-}(\omega), S_{t^-}(\omega))$, i.e. $\nu^\bF_t(\ud z)=\nu^\bF(t,X_{t^-}, S_{t^-}, \ud z)$. Furthermore, $\nu^\bF_t(\mathbb{R})=\{\nu^\bF_t(\mathbb{R}) = \eta(D_t), \ t \in [0,T]\}$  provides the $(\bF, \P)$-intensity of the point
process $\{m((0,t]\times \R), \ t \in [0,T]\}$, where $m((0,t]\times \R)$ gives the jumps number of $S$ up to time $t$.

Assume that the intensity is strictly
positive, i.e. $\eta(D_t)>0 $ $\P$-a.s. for every $t \in [0,T]$. Moreover, we also assume that
\begin{equation} \label{integrabP1}
\esp{\int_0^T  \left[|\mu_1(t, X_t, S_t)|+\sigma_1^2(t, S_t)+ \eta(D_t)+ \int_Z |K_1(\zeta;t, X_t, S_t)|\eta(\ud \zeta)\right]\ud t} < \infty,
\end{equation}
\begin{gather}
\esp{ \int_0^T  \eta(D^0_t)+ \int_Z |K_0(\zeta;t, X_t)|\eta(\ud \zeta)
\ud t} <
\infty,\label{integrabP}
\end{gather}
where $D^0_t\!:=\!\{\zeta\in Z: K_0(\zeta; t, X_{t^-} )\neq 0\}$ for every $t \in [0,T]$.

\begin{proposition}\label{prop:generatore}
Under \eqref{integrabP1} and \eqref{integrabP}, the pair $(X,S)$ is an $(\bF, \P)$-Markov process with generator $\L^{X,S}$ defined by
\begin{equation}\label{generatore}
\L^{X,S} f(t,x,s) =\frac{\partial f}{\partial t}+ \mu_1(t,x,s) s \frac{\partial f}{\partial s} +
\frac{1}{ 2} \sigma_1^2(t,s) s^2  \frac{\partial^2 f}{\partial s^2} +  \int_Z \Delta f(\zeta;t,x,s)\eta(\ud \zeta),
\end{equation}
where
\begin{equation} \label{Salti}
\Delta f (\zeta;t,x,s):= f \big (t,x+K_0(\zeta;t,x), s( 1 +K_1(\zeta; t,x, s)) \big)-f(t,x,s),
\end{equation}
for every function $f(t,x,s)$ bounded and measurable with respect to $x$ and $C^{1,2}$ with respect to $(t,s)$.
Moreover, the following semimartingale decomposition holds
\[
f(t,X_t, S_t)=f(0, x_0,s_0)+\int_0^t \L^{X,S} f(u,X_u, S_u) \ud u + M^f_t,
\]
where $M^f=\{M_t^f, \ t \in [0,T]\}$ is the $(\bF, \P)$-martingale given by

\begin{equation}\label{eq:martingale}
\ud M^f_t= \frac{\partial f}{\partial s} S_{t}\sigma_1(t,S_t)\ud W^1_t + \int_Z \Delta f(\zeta;t,X_{t^-},S_{t^-}) \left(\N(\ud t, \ud \zeta)- \eta(\ud \zeta) \ud t\right).
\end{equation}
\end{proposition}

The proof is postponed to Appendix \ref{appendix:technical}.

For simplicity in the sequel we assume
\begin{equation}\label{ass:boundedness_2}
\mu_1(t, x, s)<c_1, \quad 0<c_2<\sigma_1(t, s)<c_3 \ \mbox{ and } \ K_1(\zeta;t, x, s)<c_4,
\end{equation}
for every $(t,x,s) \in [0,T]\times \R\times \R^+, \ \zeta \in Z$ and for some constants $c_1, c_2, c_3, c_4$.

We define the process $I=\{I_t,\ t \in [0,T]\}$ by setting
\begin{equation}\label{eq:innovation}
I_t:= W^1_t+\int_0^t \frac{\mu_1(u, X_u, S_u)-\pi_u(\mu_1)}{\sigma_1(u, S_u)}\ud u, \quad t \in [0,T],
\end{equation}
where the filter $\pi$ is defined in \eqref{def:filtro}. The process $I$ is an $(\bH, \P)$-Brownian motion called the {\em innovation process} (see e.g.~\cite{FKK} and~\cite{K} for more details).

Moreover, under the assumption that
$\H_t=\F^S_t $, for every $ t\in [0,T]$, the $(\bH,\P)$-predictable dual projection of the integer-valued random measure $m(\ud t,\ud z)$ is given by $\nu_t^\bH(\ud z) \ud t$, and the following relationship holds
\[
\nu_t^{\bH}(\ud z)\ud t=\pi_{t^-}(\nu^{\bF}(\ud z))\ud t,
\]
thanks to~\cite[Proposition 2.2]{c06}.

Following the same argument of \cite[Theorem 3.1]{cco1} we get that the filtering equation with respect to  $\P$ in the jump-diffusion model is given by the following Kushner-Stratonovich equation,
\begin{equation*}
\pi_t (f) = f(0,x_0,s_0) + \int_0^t \pi_s( \L^{X,S} f) \ud s + \int_0^t h_s(f)  \ud I_s + \int_0^t \int_\R w^f(s,z) (m(\ud s, \ud
z)-\nu_s^{\bH}(\ud z)\ud s), \quad t \in [0,T],
\end{equation*}
for every function $f(t,x,s)$ bounded and measurable with respect to $x$ and $C^{1,2}$ with respect to $(t,s)$,  where
\begin{equation} \label{eq:h_jumpdiff1}
 h_t(f):=  \frac{\pi_t(\mu_1 f) - \pi_t(\mu_1) \pi_t(f)} {\sigma_1(t,S_t)} 
 + S_t \sigma_1(t,S_t) \pi_t \left(\frac{\partial f }{\partial s}\right),
 \end{equation}
\begin{equation}\label{eq:w_jumpdiff1}
w^f(t,z):= \frac{\ud \pi_{t^-} (f \nu^{\bF})}{\ud \nu^{\bH}_t} (z) - \pi_{t^-}( f) +
  \frac{ \ud \pi_{t^-} (\overline{\L} f)}{\ud \nu^{\bH}_{t}} (z),
\end{equation}
the generator $\L^{X,S}$ is given in  \eqref{generatore} and  $\overline {\L}f(t,x,s,\A):= \int_{d^\A(t,x,s)}\Delta f (\zeta;t,x,s) \eta(\ud \zeta)$,   $ \A   \in\mathcal{B}(\mathbb{R})$ and  $d^\A(t,x,s):=\{\zeta \in Z:
K_1(\zeta;t,x,s)\in \A\setminus \{0\}\}.$

Since $X$ takes value in the finite set ${\mathcal D}$, by \eqref{pi_i} we only need to compute the filtering equation for the indicator functions $f_i(x)=\I_{\{x=x_i\}}$, $i=1,\dots,d$. Then, we get

\begin{equation} \label{eq:pi_i}
\ud \pi_t(f_i)=  b_i(t, S_t, \underline{\pi_t}) \ud t + \gamma_i (t, S_t, \underline{\pi_t}) \ud I_t + \int_\R w^{f_i}(t,z) (m(\ud t, \ud
z)-\nu_t^{\bH}(\ud z)\ud t),
\end{equation}

for every $i=1,\dots,d$. Here $b_i(t,s, \underline{p})$ and $\gamma_i(t,s, \underline{p})$ are measurable functions on $[0,T] \times \R^+ \times [0,1]^d$ given by
\begin{align*}
b_i(t, S_t,  \underline{\pi_t})&= \pi_t(\L^{X,S}f_i)=\int_Z \left(\sum_{j=1}^d\pi_t(f_j) \I_{\{K_0(\zeta;t,x_j)=x_i-x_j\}}-\pi_{t^-}(f_i)\right)\eta(\ud \zeta),\\
 \gamma_i (t, S_t,  \underline{\pi_t})&= h_t(f_i)= \frac{\pi_t(f_i) \mu_1(t, x_i, S_t)- \pi_t(f_i)\sum_{j=1}^d\pi_t(f_j) \mu_1(t,x_j,S_t)}{\sigma_1(t,S_t)}
\end{align*}
for every $i=1,..., d$, and $t \in [0,T]$.

\begin{remark}
It is worth stressing that in this framework the filter can be computed recursively and turns out to be the unique solution of a linear system of $d$ equations and $d$ unknowns, see, e.g. \cite{cco2}.
\end{remark}

Now, we compute the integrand in the F\"{o}llmer-Schweizer decomposition of the random variable $\xi$ in \eqref{eq:claim} with respect to $S$, whose behavior is described in \eqref{eq:sistema}, and the filtration $\bF^S$. Note that here the process $S$ has only totally inaccessible jump times, then we can apply Corollary \ref{totally inaccessible jumps}. Firstly, this requires to construct the minimal martingale measure $\P^*$ for this model. In order to use the Ansel-Stricker Theorem (see \cite{AS}), it is necessary that the process $S$ satisfies the structure condition with respect to $\bF$.
By the dynamics \eqref{eq:sistema}, we get that the canonical decomposition of $S$ with respect to $\bF$ is
given by
\[
S_t=S_0+M_t+B^\F_t, \quad t \in [0,T],
\]
where $M$ is the square-integrable $(\bF, \P)$-martingale satisfying
\[
\ud M_t= S_t \sigma_1(t, S_t) \ud W^1_t + S_{t^-} \int_Z K_1(\zeta; t, X_{t^-}, S_{t^-} ) \widetilde \N(\ud t, \ud \zeta)=S_t \sigma_1(t, S_t) \ud W^1_t +\int_\R z (m(\ud t, \ud
z)-\nu^\bF_t(\ud z)\ud t)
\]
and $B^\F=\{B^\F,\ t \in [0,T]\}$ is the $\R$-valued $\bF$-predictable process of finite variation given by
\[
\ud B^\F_t= S_{t^-}\left\{\mu_1(t, X_{t^-}, S_{t^-} )+\int_Z K_1(\zeta;t, X_{t^-}, S_{t^-})\eta(\ud \zeta)\right\} \ud t = \left\{S_{t^-}  \mu_1(t, X_{t^-}, S_{t^-})+\int_\R z \nu^\bF_t(\ud z)\right\} \ud t.
\]
Here the $\bF$-predictable quadratic variation of $M$ is absolutely continuous with respect to the Lebesgue measure, that is,
$\ds \ud \langle M\rangle_t= a_t \ \ud t$, where
\[
a_t= S^2_{t^-} \left(\sigma_1^2(t, S_{t^-}) + \int_Z K_1^2(\zeta;t,X_{t^-}, S_{t^-})\eta(\ud \zeta)\right) =S^2_{t^-} \sigma_1^2(t, S_{t^-}) +\int_\R z^2 \nu^\bF_t(\ud z), \quad t \in [0,T].
\]

Then, the semimartingale $S$ satisfies the structure conditions with respect to $\bF$,
\[
S_t=S_0+M_t+\int_0^t\alpha^\F_s \ud \langle M\rangle_s, \quad t \in [0,T]
\]
with
\begin{gather}\label{eq:alphaF}
\alpha^\F_t= \frac{\mu_1(t, X_{t^-}, S_{t^-})+ \int_{Z}K_1(\zeta; t, X_{t^-}, S_{t^-}) \eta(\ud \zeta)}{S_{t^-}  \left(\sigma_1^2(t, S_{t^-})+\int_{Z}K_1^2(\zeta; t, X_{t^-}, S_{t^-}) \eta(\ud
\zeta)\right)}=\frac{S_{t^-} \mu_1(t, X_{t^-}, S_{t^-})+\int_\R z \ \nu^\bF_t(\ud z)}{S^2_{t^-} \sigma^2_1(t, S_{t^-}) + \int_\R z^2 \nu^\bF_t(\ud z)},
\end{gather}
for every $t \in [0,T]$. Notice that, under conditions \eqref{ass:boundedness_2}, $\alpha^\F$ is well defined and  the condition $\esp{\int_0^T(\alpha_t^\F)^2 \ud \langle M \rangle_t}<\infty$ is fulfilled.
Then, we can apply  Proposition \ref{prop:N} which provides the structure condition for $S$ with respect to $\bH$, i.e.
\[
S_t=S_0+N_t+\int_0^t \alpha^\H_s \ud \langle N\rangle_s, \quad t \in [0,T],
\]
where
\begin{gather*}
\ud N_t= S_{t} \sigma_1(t, S_t)\ud I_t + \int_\R z (m(\ud t, \ud z)- \nu_t^\bH(\ud z)\ud t), \quad \alpha^\H_t=\frac{S_{t^-} \pi_{t^-} (\mu_1) + \int_{\R}z \nu_t^\bH(\ud z)
}{S^2_{t^-} \sigma_1^2(t, S_{t^-})+\int_{\R}z^2 \nu_t^\bH(\ud z)},
\end{gather*}
for every $t \in [0,T]$.\\
To introduce the minimal martingale measure  $\P^*$, we assume that
\begin{gather}
\alpha^\F_t \Delta M_t  <1, \quad \forall \ t  \ \in [0,T], \label{eq:mmm1}\\
\label{ps3}
\bE\left[\exp\left\{\frac{1}{2}\int_0^T (\alpha^\F_t)^2 \ud \langle M^c \rangle_t + \int_0^T (\alpha^\F_t)^2 \ud \langle M^d \rangle_t\right\}\right] <
\infty.
\end{gather}
A sufficient condition for  \eqref{ps3} is given by $\ds \esp{\exp\left\{2 \int_0^T \eta(D_t)\ud t\right\}}<\infty$, (see Remark  5.6 in~\cite{ccc2}).\\
Then, we can apply the Ansel-Stricker Theorem and define an equivalent change of probability measure $\ds \left.\frac{\ud \P^*}{\ud \P}\right|_{\F_T}=L_T$, where
the process $L$, given by $\ds L_t= \doleans{- \int \alpha^\F_r \ud M_r}_t$ for every $ t \in [0,T]$, is a strictly positive  $(\bF, \P)$-martingale thanks to \eqref{ps3}. Assume that $L$ is also square integrable.
Under $\P^*$, the dynamics of the pair $(X,S)$ can be written as
\begin{equation*}
\left\{
\begin{aligned}
\ud X_t&= \int_Z K_0(\zeta; t, X_{t^-}) \N(\ud t,\ud \zeta), \quad X_0=x_0 \in  \mathcal{D},\\
\ud S_t&= S_{t^-}\left\{\sigma_1(t, S_t)\ud W^*_t + \int_Z K_1(\zeta;t,X_{t^-}, S_{t^-})\widetilde \N^*(\ud t,\ud \zeta)\right\}, \quad S_0=s
>0,
\end{aligned}
\right.
\end{equation*}
where $W^*=\{W_t^*,\ t \in [0,T]\}$ is the $(\bF, \P^*)$-Brownian motion  satisfying

\begin{equation} \label{Wstar}
W^*_t=W^1_t+\int_0^tS_u \alpha_u^\F \sigma_1(u, S_u) \ud u, \ t \in [0,T]
\end{equation}

and $\widetilde \N^*(\ud t,\ud \zeta)$  is the compensated Poisson measure under $\P^*$ given by
\[
\widetilde \N^*(\ud t,\ud \zeta)= \N(\ud t,\ud \zeta)-\eta^*_t(\ud \zeta) \ud t,
\]
with $\eta^*_t(\ud \zeta)= (1-\alpha^\F_t S_{t^-}  K_1(t, X_{t^-}, S_{t^-}))\eta(\ud \zeta)$ for every $t \in [0,T]$ and $\alpha^\F$ given in \eqref{eq:alphaF}.
We will assume that
\begin{equation} \label{integrab3}
\mathbb{E}^{\P^*}\left[ \int_0^T  \left( \eta_t^*(D^0_t)+  \eta_t^*(D_t) +  \int_Z |K_0(\zeta;t, X_t)|\eta_t^*(\ud \zeta)\right)\ud t\right] <
\infty.
\end{equation}
Since the change of probability measure is Markovian, that is, $\alpha^\F_t = \alpha^\F( t, X_{t^-}, S_{t^-})$,  the process $(X,S)$ is still an $(\bF,\P^*)$-Markov process and the following result provides the structure of its $\P^*$-generator.
\begin{proposition}\label{lemma-generatore-L3}
Under condition \eqref{integrab3}, the process $(X,S)$ is an $(\bF, \P^*)$-Markov process with generator
\begin{equation} \label{generatore-L3}
\L^*_{X,S} f(t,x,s) = \frac{\partial f}{\partial t}+ \frac{1}{ 2} \sigma_1^2(t,s)\, s^2  \frac{\partial^2 f}{\partial s^2} +  \int_Z \Delta f(\zeta;t,x,s)\eta^*_t(\ud \zeta)-\frac{\partial f}{\partial
  s} s \int_Z K_1(\zeta;t,x,s) \eta^*_t(\ud \zeta),
  \end{equation}
  and for every function $f(t,x,s)$ bounded and measurable with respect to $x$ and $C^{1,2}$ with respect to $(t,s)$.
\end{proposition}
\begin{proof}
The result follows by~\cite[Proposition 5.7]{ccc2}.
\end{proof}
We denote
by $\nu_t^{\bF,*}(\ud z) \ud t$  the $(\bF,\P^*)$-predictable dual projection of the integer-valued random measure $m(\ud t,\ud z)$
and by $\nu_t^{\bH,*}(\ud z)\ud t$ its $(\bH, \P^*)$-predictable dual
projection. Then, the relationship between $\nu_t^{\bF,*}(\ud z) \ud t$ and $\nu_t^{\bH,*}(\ud z)\ud t$ can be expressed in terms of the filter with respect to $\P^*$, $\pi^*(f)=\left\{\pi^*_t(f), \ t \in [0,T]\right\}$, defined by
\begin{equation} \label{def:filtro1}\pi^*_t(f) : =  \mathbb{E}^{\P^*}[ f(t,X_t,S_t)  | \F^S_t ] = \int_\R f(t,x,S_t) \pi^*_t(\ud x),\end{equation}
for any measurable function $f(t,x,s)$ such that $\mathbb{E}^{\P^*} |f(t,X_t,S_t)|< \infty$, for every $t\in [0,T]$.
As $\pi$, even $\pi^*$ is a probability measure valued process, which provides the conditional law of the stochastic factor $X$ given the information flow, under $\P^*$.

Therefore, thanks to~\cite[Proposition 2.2]{c06}, the following relationship holds
\[
\nu_t^{\bH,*}(\ud z)\ud t=\pi_{t^-}^*(\nu^{\bF,*}(\ud z))\ud t.
\]

\begin{remark}
According to \cite[Proposition A.2]{ccc2}, the filtering equation under the minimal martingale measure $\P^*$ is given by the following Kushner-Stratonovich equation,
\begin{equation} \label{eq:ks_jumpdiff2}
\pi^*_t (f) = f(0,x_0,s_0) + \int_0^t \pi^*_s(\L^*_{X,S} f) \ud s + \int_0^t h^*_s(f)  \ud I^*_s + \int_0^t \int_\R w^{f,*}(s,z) (m(\ud s, \ud
z)-\nu_s^{\bH,*}(\ud z)\ud s), \ \ t \in [0,T],
\end{equation}
for every function $f(t,x,s)$ bounded and measurable with respect to $x$ and $C^{1,2}$ with respect to $(t,s)$, where
\begin{equation} \label{eq:h_jumpdiff2}
 h^*_t(f)= S_t \sigma_1(t,S_t) \pi^*_t \left(\frac{\partial f }{\partial
 s}\right),
 \end{equation}
\begin{equation}\label{eq:w_jumpdiff2}
w^{f,*}(t,z)= \frac{\ud \pi^*_{t^-} (f \nu^{\bF,*})}{\ud \nu^{\bH,*}_t} (z) - \pi^*_{t^-}(f) +
  \frac{ \ud \pi_{t^-}^* (\overline{\L} f)}{\ud \nu^{\bH,*}_{t}} (z),
\end{equation}
and $\L^*_{X,S}$ is given in \eqref{generatore-L3}.

The process $I^*=\{I^*_t, \ t \in [0,T]\}$ is the $(\bH, \P^*)$-Brownian motion given by
\begin{equation} \label{I*}
I^*_t=W^*_t+ \int_0^t\left\{\frac{b(u, X_u, S_u)}{\sigma_1(u, S_u)} - \pi^*_u\left(\frac{b}{\sigma_1}\right)\right\}\ud u,\quad t \in [0,T],
\end{equation}
with $\ds b(t, X_t, S_t)= \int_Z K_1(\zeta; t , X_t, S_t) \ \eta_t^*(\ud \zeta)$, for every $t \in [0,T]$.

Since $X$ assumes finitely many values, we can also characterize the filter in terms of the conditional probabilities under $\P^*$ with the choice $f(x)=f_i(x) = 1_{\{x=i\}}$, $i=1,..,d$. Denoting by $\underline{\pi^*}:=\{\underline{\pi^*_t}, t \in [0,T]\}$ the vector process $\underline{\pi^*_t}=(\pi^*_t(f_1),\dots,\pi_t^*(f_d))$, we get
\begin{equation} \label{KSfilterF*}
\pi^*_t (f_i) = f_i(0,x_0,s_0) + \int_0^t b^*_i(s, S_s,  \underline{ \pi^*_s}) \ud s + \int_0^t \int_\R w^{f_i,*}(s,z) (m(\ud s, \ud
z)-\nu_s^{\bH,*}(\ud z)\ud s),
\end{equation}
where $b^*_i(t,s, \underline{p})$ are measurable functions on $[0,T] \times \R^+\times [0,1]^d$ satisfying
\begin{align*}
b^*_i(t, S_t,  \underline{\pi^*_t})&= \int_Z \left(\sum_{j=1}^d\pi^*_t(f_i) \I_{\{K_0(\zeta;t,x_j)=x_i-x_j\}}-\pi^*_{t}(f_i)\right)\eta^*_t(\ud \zeta)
\quad i=1,..., d.\end{align*}
\end{remark}

The following result shows that the vector $(X,S,\underline{\pi})$ is an $(\bF, \P^*)$-Markov process and provides its $\P^*$-generator.

\begin{proposition}\label{Markov}
Assume that condition \eqref{integrab3} holds. Then, the vector $(X,S, \underline{\pi})$ is an $(\bF, \P^*)$-Markov process with generator $\L^*$ given by
\begin{align}
\L^* f(t,x,s, \underline{p}) &= \frac{\partial f}{\partial t}+  l(t,x,s,\underline{p}) \sum_{i=1}^d \gamma_i(t,s,\underline{p}) \frac{\partial f}{\partial p_i} + \sum_{i=1}^d b_i(t,s,\underline{p}) \frac{\partial f}{\partial p_i}\nonumber\\
&-\sum_{i=1}^d  \int_\R \frac{\partial f}{\partial p_i}w_i(t,x,s,\underline{p},z)\nu^\bH(t, x,s,\underline{p},\ud z) + \frac{1}{ 2} \sigma^2_1(t,s)\, s^2  \frac{\partial^2 f}{\partial s^2}  \nonumber \\
&+  \sum_{i=1}^d \sigma_1(t,s)\, s  \gamma_i(t,s,\underline{p}) \frac{\partial^2 f}{\partial s \partial p_i }+ \frac{1}{ 2} \sum_{i,j=1}^d \gamma_i(t,s,\underline{p})  \gamma_j(t,s,\underline{p}) \frac{\partial^2 f}{\partial p_i \partial p_j }\nonumber\\
& + \int_Z \Delta f(\zeta;t,x,s,  \underline{p})\eta^*(t,x,s,\ud \zeta)-
\frac{\partial f}{\partial  s} s \int_Z K_1(\zeta;t,x,s) \eta^*(t,x,s,\ud \zeta), \label{generatore-L4}
\end{align}
for every function $f(t,x,s,\underline{p})$ bounded and measurable with respect to $x$ and $C^{1,2,2}$ with respect to $(t,s, \underline{p})$, where $w_i(t,x,s,\underline{p},z)$ is the measurable function such that $w_i(t, X_{t^-}, S_{t^-}, \underline{\pi_{t^-}},z)=w^{f_i}(t,z)$ with $w^{f_i}(t,z)$ given in \eqref{eq:w_jumpdiff1} with the choice $f(t,x,s)=f_i(x)=\I_{\{x=x_i\}}$,
\begin{align*}
&l(t,x,s,\underline{p}):= \frac{\mu_1(t, x,s)-\sum_{i=1}^d \mu_1(t, x_i,s) p_i}{\sigma_1(t, s)} - s \sigma_1(t, s) \alpha^\F(t,x,s),\\
&\eta^*(t,x,s,\ud\zeta):=[1-\alpha^\F(t,x,s) \ s \ K_1(t,x,s)]\eta(\ud \zeta),\\
&\nu^\bH(t,x,s,\underline{p}, \ud z):=\sum_{i=1}^k\nu^\bF(t,x_i,s,\ud z) p_i,\\
&\Delta f(\zeta;t,x,s, \underline{p}):=  f \big (t,x+K_0(\zeta;t,x), s( 1 +K_1(\zeta; t,x, s)), \underline{p} + \Delta \underline{p}(\zeta; t,x, s) \big)-f(t,x,s, \underline{p}),\\
&\Delta \underline{p}(\zeta; t,x, s) := ( \Delta p_1, \dots , \Delta p_d),
\end{align*}
and
$$  \Delta p_i := w^{f_i}( t, s  \ K_1(\zeta; t,x, s)) \ \I_{\{ K_1(\zeta; t,x, s) \neq 0\}}(\zeta), \quad i=1,..., d.$$
\end{proposition}

The proof is postponed to Appendix \ref{appendix:technical}.

Finally, the following theorem provides the explicit expression of the integrand in the F\"{o}llmer-Schweizer decomposition under partial information of $\xi$, given in \eqref{eq:claim},  with respect to $S$ and  $\bF^S$ for the Markovian model described in this section.
In the sequel we use the following notation:
\begin{gather}
g^i_t=g(t, x_i, S_t, \underline{\pi_t}), \quad g^i_{t^-}=g(t, x_i, S_{t^-}, \underline{\pi_{t^-}})\\
 \Delta g^i_t(z)=\Delta g(z;t, x_i, S_t, \underline{\pi_t})=g\left(t, x_i, S_{t^-}+z, (\pi^i_{t^-}+w^{f_i}(t,z))_{i=1,\dots,d}\right)-g(t, x_i, S_{t^-}, \underline{\pi_{t^-}})
\end{gather}

\begin{theorem}\label{thm:IntegrandFS}
Let $g(t, x, s, \underline{p} )$ be a measurable function on  $[0,T] \times \R\times \R^+\times[0,1]^d$ that solves problem \eqref{eq:problema1}. Then the integrand in F\"{o}llmer-Schweizer decomposition of $\xi$, see \eqref{eq:claim}, with respect to $\bF^S$ and $S$ is given by
\[
\beta^\H_t= H^\H_t+\phi^\H_t, \quad t \in [0,T],
\]
where

\begin{equation}\label{eq:Hh}
\begin{split}
 H^{\H}_t= & \frac{\sum_{i=1}^d S_{t^-} \sigma_1(t, S_{t^-})\pi_{t^-}^*(f_i)\left(\frac{\partial g^i_{t^-}}{\partial s} S_{t^-} \sigma_1(t, S_{t^-}) + \sum_{j=1}^d\frac{\partial g^i_{t^-}}{\partial p^j}\gamma_j(t, S_{t^-}, \underline{\pi_{t^-}})\right)}{S^2_{t^-} \sigma^2_1(t,S_{t^-})  +  \int_\R  z^2 \nu _t^{\bH,*} (\ud z)}\\
& + \frac{\sum_{i=1}^d \int_\R \left\{\Delta g^i_t(z)(\pi_{t^-}^*(f_i)+w^{f_i,*}(t,z)) + g^i_{t^-} w^{f_i,*}(t,z)\right\} z \ \nu _t^{\bH,*} (\ud z)} {S^2_{t^-} \sigma^2_1(t,S_{t^-})  +  \int_\R  z^2 \nu _t^{\bH,*} (\ud z)},
\end{split}
\end{equation}

\begin{equation} \label{eq:phih}
\phi^\H_t = \frac{   \int_\R  \left\{ \sum_{i=1}^d \int_\R \left\{\Delta g^i_t(z)(\pi_{t^-}^*(f_i)+w^{f_i,*}(t,z)) + g^i_{t^-} w^{f_i,*}(t,z)\right\} z  - H^\H_t  z^2 \right\} z \alpha_t^\H \nu^\bH_t(\ud z)}  {S^2_{t^-} \sigma^2_1(t,S_{t^-})  +  \int_\R  z^2 \nu _t^{\bH} (\ud z)},
 \end{equation}
for every $t \in [0,T]$.
\end{theorem}

\begin{proof}
Notice that the $(\bH, \P^*)$-semimartingale decomposition of $S$  is given by
\begin{equation}\label{eq:Hsemimartingala}
\ud S_t = S_t \sigma_1(t, S_t) \ud I^*_t + \int_\R z ( m(\ud t, \ud z) - \nu_t ^{\bH,*} (\ud z) \ud t ).
\end{equation}
Now, we recall that
$$
V_t^\H = \bE^{\P^*}\left[ g(t, X_t, S_t, \underline{\pi_t})   \Big| \H_t\right] = \sum_{i=1}^d   g^i_t \pi^*_t(f_i)\quad t \in [0,T].  $$

To compute the dynamics of $V^\H$ we first apply the It\^{o} formula to $g(t, x_i, S_t, \underline{\pi_t})$. Taking the dynamics of
$\underline{\pi}$ (see \eqref{eq:pi_i}) and formulas \eqref{eq:innovation}, \eqref{Wstar} and \eqref{I*} into account, we get

\begin{equation}
\ud g^i_t=h^i_t\ud t+\left\{\frac{\partial g^i_t}{\partial s}S_t \ \sigma_1(t, S_t)+\sum_{j=1}^d\frac{\partial g^i_t}{\partial p^j}\gamma_j(t, S_t, \underline{\pi_t})\right\}\ud I^*_t+\int_{\R}\Delta g^i_t(z)(m(\ud t, \ud z)-\nu^{\bH, *}(\ud z)\ud t)
\end{equation}
for a suitable $\bH$-predictable process $h^i=\{h^i_t, \ t \in [0,T]\}$.

In particular, since the function $g$ is the solution of problem \eqref{eq:problema1}, the dynamics of $V^\H$ is given by
\begin{align*}
\ud V_t^\H =&\sum_{i=1}^d\pi_t^*(f_i)\left(\frac{\partial g^i_t}{\partial s} S_t \sigma_1(t, S_t)+\sum_{j=1}^d\frac{\partial g^i_t}{\partial p^j}\gamma_j(t, S_{t}, \underline{\pi_t})\right)\ud I^*_t  \\
  & +\sum_{i=1}^d \int_Z \left\{\Delta g^i_t(z) \Big(\pi_{t^-}^*(f_i)+w^{f_i,*}(t,z)\Big) + g^i_{t^-} w^{f_i,*}(t,z) \right\} (m(\ud t, \ud z)-\nu_t^{\bH, *}(\ud z)\ud t).
  \end{align*}

Then, the integrand in the Galtchouk-Kunita-Watanabe decomposition of ${}^o\xi$ under $\P^*$ is given by
\begin{align}
H^\H_t& =  \frac{\ud {}^{*,\bH}\langle V^\H, S\rangle_t}{\ud {}^{*,\bH} \langle S\rangle_t},\quad t \in [0,T]
\end{align}
which yields \eqref{eq:Hh}. In order to compute the integrand $\beta^\H$ in the F\"{o}llmer-Schweizer decomposition of $\xi$, we observe that
$$\ud G_t = \ud V^\H_t - H^\H_t \ud S_t$$
and
\begin{align*}
{}^\bH\left\langle [G, S], \int_0^\cdot \alpha_s^\H \ud N_s \right\rangle_t
= {}^\bH\left\langle \sum_{r\leq t} \Delta G_r \Delta S_r , \int_0^t \alpha_s^\H \int_\R z ( m(\ud r, \ud z) - \nu_r^\bH (\ud z) \ud r  ) \right\rangle.
\end{align*}
Finally, by Corollary \ref{totally inaccessible jumps}
we get that $ \beta^\H_t= H^\H_t+\phi^\H_t,$ where $\phi^\H_t = \frac{\ud {}^\bH\langle [G, S], \int_0^\cdot \alpha_s^\H \ud N_s \rangle_t}{\ud {}^\bH \langle N \rangle_t}$ is given by \eqref{eq:phih}.
\end{proof}

\begin{remark}
Note that, if the process $S$ has continuous trajectories, then by Theorem \ref{thm:IntegrandFS} we get
\begin{equation}
\beta_t^\H = H_t^\H =\sum_{i=1}^d \pi_{t^-}^*(f_i)\frac{\partial g^i_{t^-}}{\partial s}  + \frac{\sum_{i=1}^d\sum_{j=1}^d\pi_{t^-}^*(f_i)\frac{\partial g^i_{t^-}}{\partial p^j}\gamma_j(t, S_{t^-}, \underline{\pi_{t^-}})}{S_{t^-} \sigma_1(t,S_{t^-})}.
\end{equation}

In fact, in this case the F\"ollmer-Schweizer decomposition under $\P$ of a given random variable coincides with the corresponding Galtchouk-Kunita-Watanabe decomposition under $\P^*$, see \cite[Theorem 3.5]{s01}.

If in addition, the random variable $\xi$ turns out to be $\H_T$-measurable, then
\begin{equation} \label{eq:Hp}
\beta_t^\H = H_t^\H= \sum_{i=1}^d \pi_{t^-}^*(f_i)\frac{\partial g^i_{t^-}}{\partial s}  = {}^{p,*}\beta^\F_t,
\end{equation}
where ${}^{p,*}\beta^\F$ corresponds to the $(\bH,\P^*)$-predictable projection of the process $\beta^\F=\{\beta^\F_t, \ t \in [0,T]\}$, where $\beta_t^\F =  \frac{\partial g(t, X_{t^-}, S_{t^-})}{\partial s} $, and $g(t,x,s)$ is characterized in \cite[Lemma 5.1]{ccc2}.
The process $\beta^\F$ represents the integrand in the F\"ollmer-Schweizer decomposition of $\xi$ under complete information, that coincides with the Galtchouk-Kunita-Watanabe decomposition of $\xi$ with respect to $\bF$ and $S$ under the minimal martingale measure $\P^*$. We remark that relation \eqref{eq:Hp} also holds thanks to \cite[Proposition 4.6]{ccc2}.
\end{remark}

\bibliographystyle{plain}
\bibliography{biblio2}

\appendix
\section{Technical results}\label{appendix:technical}

\begin{proof}[Proof of Proposition \ref{prop:generatore}]

By applying It\^{o}'s formula to the function $f(t,X_t,S_t)$, we get
\[
f(t, X_t, S_t)= f(0, x_0, s_0) + \int_0^t \L^{X,S}f (r,X_r,S_r)\ud r + M^f_t,
\]
where $\ds \L^{X,S}$ is the operator in \eqref{generatore} and $M^f$ is given in \eqref{eq:martingale}.
Moreover, under conditions \eqref{integrabP1} and \eqref{integrabP},
the process $M^f$ is an $(\bF, \P)$-martingale; indeed
\[
\bE\left[
\int_0^T  \sigma_1^2(t,S_t) S^2_t \left( \frac{\partial f}{\partial s} \right)^2 \ud t\right] < \|f\|\bE\left[
\int_0^T  \sigma_1^2(t,S_t) S^2_t \ud t\right]<\infty
\]
and
\[
\bE\left[\int_0^T \int_Z |\Delta f(\zeta;t,X_{t^-},S_{t^-})| \eta(\ud \zeta)\;\ud t\right]
\leq 2 \|f\| \bE\left[\int_0^T \{ \eta(D^0_t) + \eta(D_t)\} \ud t\right]  < \infty ,
\]
where $\|f\|=\sup\{f(t,x,s)+\frac{\partial f}{\partial s}(t,x,s)|(t,x,s) \in  \R^+ \times \R \times \R^+\}$.  Then the pair $(X,S)$ is a solution of the martingale problem for the operator $\L^{X,S}$ which implies that $(X,S)$ is an $(\bF, \P)$-Markov process.
\end{proof}

\begin{proof}[Proof of Proposition \ref{Markov}]
By applying It\^{o}'s formula to the function $f(t,X_t,S_t, \underline{\pi_t})$, then we get that
\[
f(t, X_t, S_t, \underline{\pi_t})= f(0, x_0, s_0, \underline{\pi_0}) + \int_0^t \L^*f (r,X_r,S_r, \underline{\pi_r})\ud r + M^{*,f}_t,
\]
where $\ds \L^*$ is the operator in \eqref{generatore-L4} and $M^{*,f}$ is given by
\[
\ud M^{*,f}_t= \left(\frac{\partial f}{\partial s}(t,X_t,S_t, \underline{\pi_t})S_t\sigma_1(t, S_t) + \sum_{j=1}^d\frac{\partial f}{\partial p_i} \gamma_i(t, S_t, \underline{\pi_t}) \right)\ud W^*_t + \int_Z \Delta f(\zeta;t,X_{t^-},S_{t^-}, \underline{\pi_{t^-}})\widetilde{\mathcal{N}}^*(\ud t, \ud \zeta).
\]
Then, arguing as in the proof of Proposition \ref{prop:generatore}, we get the statement.
\end{proof}

\end{document}